\documentclass{gtpart}
\usepackage{pinlabel}
\usepackage{graphicx}
\title[Transverse properties of parabolic subgroups of Garside groups]{Transverse properties of parabolic subgroups of Garside groups}
\author[Y Antol\'in]{Yago Antol\'in}
\givenname{Yago}
\surname{Antol\'in}
\address{Departamento de Matem\'aticas, Facultad de Ciencias, Universidad Aut\'onoma de Madrid, 28049 Madrid, Spain}
\email{yago.anpi@gmail.com}	
\urladdr{}
\author[L Paris]{Luis Paris}
\givenname{Luis}
\surname{Paris}
\address{IMB, UMR 5584, CNRS, Univ. Bourgogne Franche-Comt\'e, 21000 Dijon, France}
\email{lparis@u-bourgogne.fr}
\urladdr{}
\subject{primary}{msc2000}{20F36, 20F10}
\volumenumber{}
\issuenumber{}
\publicationyear{}
\papernumber{}
\startpage{}
\endpage{}
\doi{}
\MR{}
\Zbl{}
\received{}
\revised{}
\accepted{}
\published{}
\publishedonline{}
\proposed{}
\seconded{}
\corresponding{}
\editor{}
\version{}
\newtheorem{thm}{Theorem}[section]
\newtheorem{lem}[thm]{Lemma}
\newtheorem{prop}[thm]{Proposition}
\newtheorem{corl}[thm]{Corollary}

\theoremstyle{definition}
\newtheorem*{rem}{Remark}

\numberwithin{equation}{section}

\makeatletter
\renewcommand{\thefigure}{\ifnum \c@section>\z@ \thesection.\fi
 \@arabic\c@figure}
\@addtoreset{figure}{section}
\makeatother

\begin{document}

\def\N{\mathbb N} \def\Div{{\rm Div}} \def\SS{\mathcal S}
\def\ev{{\rm ev}} \def\Gr{{\rm Gr}} \def\Geo{{\rm Geo}}
\def\Z{\mathbb Z} \def\FF{\mathcal F} \def\AA{\mathcal A}
\def\diam{{\rm diam}} \def\Min{{\rm Min}}

%%%%%%%%%%

\begin{abstract}
Let $G$ be a Garside group endowed with the generating set $\SS$ of non-trivial simple elements, and let $H$ be a parabolic subgroup of $G$.
We determine a transversal $T$ of $H$ in $G$ such that each $\theta \in T$ is of minimal length in its right-coset, $H \theta$, for the word length with respect to $\SS$.
We show that there exists a regular language $L$ on $\SS \cup \SS^{-1}$ and a bijection $\ev : L \to T$ satisfying $\lg (U) = \lg_\SS( \ev(U))$ for all $U \in L$.
>From this we deduce that the coset growth series of $H$ in $G$ is rational. 
Finally, we show that $G$ has fellow projections on $H$ but does not have bounded projections on $H$. 
\end{abstract}

\maketitle

%%%%%%%%%%

\section{Introduction}\label{sec1}

Let $S$ be a finite set. 
A \emph{Coxeter matrix} over $S$ is defined to be a square matrix $M =(m_{s,t})_{s,t \in S}$ indexed by the elements of $S$, with coefficients in $\N \cup \{ \infty \}$, such that $m_{s,s}=1$ for all $s \in S$, and $m_{s,t} = m_{t,s} \ge 2$ for all $s,t \in S$, $s \neq t$.
For each $s,t \in S$, $s \neq t$, and each integer $m \ge 2$, we set $\Pi(s,t,m) = (st)^{\frac{m}{2}}$ if $m$ is even and $\Pi(s,t,m)= (st)^{\frac{m-1}{2}} s$ if $m$ is odd. 
In other words, $\Pi(s,t,m)$ denotes the word $sts \cdots$ of length $m$.
The \emph{Artin group} associated with the above Coxeter matrix  $M$ is the group $A=A_M$ defined by the presentation 
\[
A = A_M = \langle S \mid \Pi (s,t,m_{s,t}) = \Pi(t,s,m_{s,t}) \text{ for all } s,t \in S,\ s \neq t \text{ and } m_{s,t} \neq \infty \rangle\,.
\]
The \emph{Coxeter group} associated with $M$, denoted by $W=W_M$, is the quotient of $A$ by the relations $s^2=1$, $s \in S$.

Artin groups were introduced by Tits \cite{Tits1} as extensions of Coxeter groups. 
There are few known results valid for all Artin groups, and most of the work in the subject concerns particular families of them.
The family concerned by the present paper is the one of Artin groups of spherical type, that is, the Artin groups whose associated Coxeter groups are finite. 
Seminal work on these groups are Brieskorn \cite{Bries1,Bries2}, Brieskorn--Saito \cite{BriSai1} and Deligne \cite{Delig1} in relation with the so-called discriminant varieties.
They are involved in several subjects such as singularities, Hecke algebras, hyperplane arrangements, mapping class groups, and there is an extensive literature on them. 
The leading examples of spherical type Artin groups are the braid groups.

Let $G$ be a group and let $M$ be a submonoid of $G$ such that $M \cap M^{-1} = \{1\}$.
Then we have two partial orders $\le_L$ and $\le_R$ on $G$ defined by $\alpha \le_L \beta$ if $\alpha^{-1} \beta \in M$, and $\alpha \le_R \beta$ if $\beta\alpha^{-1} \in M$.
For each $a \in M$ we set $\Div_L(a) = \{b \in M \mid  b \le_L a\}$ and $\Div_R(a) = \{b \in M \mid b \le_R a\}$.
An element $a \in M$ is called \emph{balanced} if $\Div_L(a) = \Div_R(a)$.
In this case we set $\Div(a) = \Div_L(a) = \Div_R(a)$.
On the other hand, we say that $M$ is \emph{Noetherian} if, for each $a \in M$, there exists an integer $n \ge 1$ such that $a$ cannot be decomposed as a product of more than $n$ non-trivial factors.

Let $G$ be a group, let $M$ be a submonoid of $G$ such that $M \cap M^{-1} = \{1\}$, and let $\Delta$ be a balanced element of $M$.
We say that $G$ is a \emph{Garside group} with \emph{Garside structure} $(G,M,\Delta)$ if:
\begin{itemize}
\item[(a)]
$M$ is Noetherian;
\item[(b)]
$\Div(\Delta)$ is finite, it generates $M$ as a monoid, and it generates $G$ as a group;
\item[(c)]
$(G, \le_L)$ is a lattice. 
\end{itemize}
In this case $\Delta$ is called the \emph{Garside element} and the elements of $\Div(\Delta)$ are called the \emph{simple elements} of $G$ for the given Garside structure. 
The lattice operations of $(G,\le_L)$ are denoted by $\wedge_L$ and $\vee_L$, where for $\alpha, \beta \in G$,
$$ \alpha \wedge_L \beta =\max_{\le_L}(\gamma\in G \mid \gamma \le_L \alpha \text{ and } \gamma \le_L \beta) $$
and 
$$ \alpha \vee_L \beta =\min_{\le_L}(\gamma \in G \mid \alpha \le_L \gamma \text{ and } \beta \le_L \gamma)\,.$$
Remark that if $\alpha, \beta \in M$, then $1\le_L \alpha$  and $1 \le_L \beta$ and therefore $1 \le_L (\alpha \wedge_L \beta) \le_L (\alpha \vee_L \beta)$ and thus, $\wedge_L$ and $\vee_L$ restrict to lattice operations of $(M, \le_L)$.
On the other hand, the ordered set $(G, \le_R)$ is also a lattice and its lattice operations are denoted by $\wedge_R$ and $\vee_R$ and these operations also restrict to lattice operations of $(M, \le_R)$.

This definition is due to Dehornoy--Paris \cite{DehPar1}. 
It is inspired by the work of Garside \cite{Garsi1}, Epstein et al. \cite{Epste1}, Charney \cite{Charn1,Charn2} and others on braid groups and, more generally, on Artin groups of spherical type. 
It isolates certain combinatorial characteristics of the braid groups and of the spherical type Artin groups that suffice to show certain algebraic properties such as the existence of a solution to the conjugacy problem, or the bi-automaticity.  
By Brieskorn--Saito \cite{BriSai1} and Deligne \cite{Delig1}, the Artin groups of spherical type are Garside groups, and a large part of the study of these groups is made now in terms of Garside groups.
We refer to Dehornoy et al. \cite{DehEtAl} for a full account on the theory.

Let $A$ be an Artin group associated with a Coxeter matrix $M=(m_{s,t})_{s,t \in S}$.
The subgroup $A_X$ of $A$ generated by a subset $X$ of $S$ is called a (standard) \emph{parabolic subgroup} of $A$.
By van der Lek \cite{Lek1} we know that $A_X$ is itself an Artin group, associated with the Coxeter matrix $M_X=(m_{s,t})_{s,t \in X}$.
Parabolic subgroups play a prominent role in the study of Artin groups in general and in that of Artin groups of spherical type in particular.
For example, they are one the main ingredients in the definition/construction of the Salvetti complex for Artin groups, which is one of the most important tools in the study of the $K(\pi,1)$ conjecture for Artin groups (see for example Godelle--Paris \cite{GodPar1} and Paris \cite{Paris1}) and in the calculation of the cohomology of these groups (see for example De Concini--Salvetti \cite{DeCSal1} and Callegaro--Moroni--Salvetti \cite{CaMoSa1}).

The notion of parabolic subgroup was extended to Garside groups by Godelle \cite{Godel1, Godel2} as follows. 
Let $(G,M,\Delta)$ be a Garside structure on a Garside group $G$ and let $\delta$ be a balanced element in $M$.
Let $G_\delta$ (resp. $M_\delta$) denote the subgroup of $G$ (resp. the submonoid of $M$) generated by $\Div(\delta)$.
Then we say that $(G_\delta,M_\delta,\delta)$ is a \emph{parabolic substructure} if $\Div(\delta) = \Div(\Delta) \cap M_\delta$.
Note that this condition implies that $\delta \in \Div (\Delta)$.
In this case we say that $G_\delta$ is a \emph{parabolic subgroup} of $G$ and that $M_\delta$ is a \emph{parabolic submonoid} of $M$.

In this paper a \emph{parabolic subgroup will be always assumed to be non-trivial.}

Let $A$ be an Artin group of spherical type, let $S$ be the standard generating set of $A$, and let $X$ be a subset of $S$.
Then, by Brieskorn--Saito \cite{BriSai1}, the subgroup $A_X$ of $A$ generated by $X$ is a parabolic subgroup in Godelle's sense, and all parabolic subgroups of $A$ are of this form.

\begin{rem}
Let $H$ be a parabolic subgroup of a Garside group $G$ with Garside structure $(G, M, \Delta)$. 
Then there is a unique parabolic substructure $(G_\delta, M_\delta, \delta)$ such that $H=G_\delta$.
Indeed, the balanced element $\delta$ which defines this parabolic substructure is the greatest element of $\Div(\Delta)\cap H$ for the ordering $\le_L$, hence $H$ determines $\delta$.
Similarly, if $N$ is a parabolic submonoid of $M$, then there is a unique parabolic substructure $(G_\delta, M_\delta, \delta)$ such that $N=M_\delta$, where $\delta$ is the greatest element of $\Div(\Delta) \cap N$ for the ordering $\le_L$.
So, we can speak of a parabolic subgroup and/or of a parabolic submonoid without necessarily mentioning the balanced element $\delta$ or the triple $(G_\delta,M_\delta,\delta)$.
\end{rem}

The results of the present paper are new for Artin groups of spherical type and even for braid groups and, as often now in the theory, the whole study is done in terms of Garside groups.

Let $G$ be a Garside group with Garside structure $(G, M, \Delta)$.
As often in the theory the generating set of $G$ that will be considered in the present paper is the set $\SS=\Div (\Delta) \setminus \{1\}$ of non-trivial simple elements. 
For each $\alpha \in G$ we denote by $\lg(\alpha) = \lg_\SS(\alpha)$ the word length of $\alpha$ with respect to $\SS$.
Let $H$ be a parabolic subgroup of $G$.
Recall that a \emph{right-coset} of $H$ in $G$ is a subset of $G$ of the form $H \alpha = \{ \beta \alpha \mid \beta \in H\}$, with  $\alpha \in G$.
Recall also that a (right) \emph{transversal} of $H$ in $G$ is a subset $T$ of $G$ such that $C \cap T$ is a singleton for every right-coset $C$ of $H$ in $G$. 
On the other hand, we define the \emph{length} of a right-coset $C$ of $H$ in $G$ to be $\lg(C) = \min\{\lg(\beta) \mid \beta \in C\}$.
In a first step we will determine an explicit transversal $T$ of $H$ in $G$ which satisfies $\lg(\theta) = \lg(H\theta)$ for all $\theta \in T$ (see Theorem \ref{thm3_3} and Theorem \ref{thm3_4}).

We denote by $\ev : (\SS \cup \SS^{-1})^* \to G$ the map which sends each word $U$ to the element of $G$ that it represents. 
In a second step we will determine a regular language $L$ over $(\SS \cup \SS^{-1})^*$ satisfying $\ev(L) = T$, the restriction $\ev : L \to T$ is a bijective correspondence, and $\lg(U) = \lg(\ev (U))$ for all $U \in L$ (see Theorem \ref{thm4_2}).
The notion of regular language will be recalled in Section \ref{sec4}.

For each $n \in \N$ we denote by $e(n) = e_{G,H,\SS}(n)$ the number of right-cosets of $H$ in $G$ of length $n$.
The \emph{coset growth series} of $H$ in $G$ is defined to be the formal series $\Gr_{G,H,\SS} (t) = \sum_{n=0}^\infty e(n) t^n$.
By the above $e(n)$ is the number of words of length $n$ in the regular language $L$.
So, $\Gr_{G,H,\SS} (t)$ is equal to the growth series of $L$.
As we know that the growth series of a regular language is rational (see Flajolet--Sedgewick \cite{FlaSed1} for example), the formal series $\Gr_{G,H,\SS} (t)$ is rational.

The starting point of this work was the paper \cite{Antol1} where the first author studies the triples $(G,\SS,H)$ where $G$ is a group, $\SS$ a finite generating set of $G$ and $H$ is a subgroup of $G$ such that $(G,\SS)$ has the falsification by fellow traveller property and has fellow projections and/or bounded projections on $H$. 
These notions will be recalled in Section \ref{sec5}.
In particular, it is proved in Antol\'in \cite{Antol1} that, if $(G,\SS)$ has the falsification by fellow traveller property and has fellow projections on $H$, then the language $\Geo (H \backslash G, \SS) = \{U \in (\SS \cup \SS^{-1})^* \mid  \lg(U) = \lg(H\, \ev (U)) \}$ is regular. 
It is also proved in Antol\'in \cite{Antol1} that, if $(G,\SS)$ has the falsification by fellow traveller property and has bounded projections on $H$, then the coset growth series $\Gr_{G,H,\SS} (t)$ is rational. 
Let $G$ be a Garside group, $\SS=\Div(\Delta)\setminus\{1\}$ and let $H$ be a parabolic subgroup of $G$.
We know by Holt \cite{Holt1} that $(G,\SS)$ has the falsification by fellow traveller property.
So, it is natural to ask whether $G$ has fellow projections and/or bounded projections on $H$.

In the present paper we show that $(G,\SS)$ has fellow projections but does not have bounded projections on $H$ (if $H \neq G$) (see Corollary \ref{corl5_4} and Theorem \ref{thm5_5}).
This is the first known example of a triple $(G,\SS,H)$ where $(G,\SS)$ has the falsification by fellow traveller property, has fellow projections on $H$, and does not have bounded projections on $H$. 
On the other hand, note that, by the results of Section \ref{sec4}, the coset growth series $\Gr_{G,H,\SS} (t)$ is rational in our case, and we do not know if there exists such an example with a non-rational coset growth series.
Note also that the transversal $T$ defined and studied in Section \ref{sec3} is an important tool in the proof of Theorem \ref{thm5_5}.

%%%%%%%%%%

\section{Preliminaries}\label{sec2}

Let $G$ be a Garside group with Garside structure $(G, M, \Delta)$.
Recall that $\SS = \Div (\Delta) \setminus \{1\}$ is the generating system of $G$ that will be considered in the present paper and $\lg = \lg_\SS$ denotes the word length with respect to $\SS$. 
The \emph{left greedy normal form} of an element $a \in M$ is defined to be the unique expression $a = u_1 u_2 \cdots u_p$ of $a$ over $\SS$ such that $u_p \neq 1$ and $(u_i \cdots u_p) \wedge_L \Delta = u_i$ for all $i \in \{1, \dots, p\}$.
We define the \emph{right greedy normal form} of $a$ in a similar way. 
The following two theorems contain fundamental results on Garside groups.

\begin{thm}[Dehornoy--Paris \cite{DehPar1}, Dehornoy \cite{Dehor1}]\label{thm2_1}
\begin{itemize}
\item[(1)]
Let $a \in M$, let $a = u_1 
\allowbreak
u_2 \cdots u_p$ be its left greedy normal form and let $a=u_q' \cdots u_2'u_1'$ be its right greedy normal form.
Then $p=q = \lg (a)$.
\item[(2)]
Let $\alpha \in G$.
There exists a unique pair $(a,b) \in M \times M$ such that $\alpha = b^{-1} a$ and $a \wedge_L b = 1$.
 Similarly, there exists a unique pair $(a',b') \in M \times M$ such that $\alpha = a'b'^{-1}$ and $a' \wedge_R b'=1$.
In this case we have $\lg (\alpha) = \lg(a) + \lg(b) = \lg(a') + \lg(b')$.
\end{itemize}
\end{thm}

The expressions $\alpha = b^{-1} a$ and $\alpha = a' b'^{-1}$ given in Theorem \ref{thm2_1}\,(2) are called \emph{left orthogonal form} of $\alpha$ and \emph{right orthogonal form} of $\alpha$, respectively. 
Let $\alpha \in G$ and let $\alpha = b^{-1} a$ be its left orthogonal form.
Let $a = u_1 u_2 \cdots u_p$ and $b = v_1 v_2 \cdots v_q$ be the left greedy normal forms of $a$ and $b$, respectively.
Then the expression $\alpha = v_q^{-1} \cdots v_2^{-1} v_1^{-1} u_1 u_2 \cdots u_p$ is called the \emph{left greedy normal form} of $\alpha$.
Note that, by Theorem \ref{thm2_1}, $\lg (\alpha) = p+q$.
We define the \emph{right greedy normal form} of an element of $G$ in a similar way.

An element $a \in M$ is called \emph{unmovable} if $\Delta \not \le_L a$ or, equivalently, if $\Delta \not \le_R a$.
On the other hand, we denote by $\Phi : G \to G$, $\alpha \mapsto \Delta \alpha \Delta^{-1}$, the conjugation by $\Delta$.
Note that $\Phi (\SS) = \SS$ and $\Phi (M) = M$, and therefore $\Phi$ preserves $\le_L$ and $\le_R$ and thus, it commutes with  $\wedge_L$, $\vee_L$, $\wedge_R$ and $\vee_R$.

\begin{thm}[Dehornoy--Paris \cite{DehPar1}, Dehornoy \cite{Dehor1}]\label{thm2_2}
Let $\alpha \in G$.
There exists a unique pair $(a,k) \in M \times \Z$ such that $a$ is unmovable and $\alpha = \Delta^k a$.
Similarly, there exists a unique pair $(a',k') \in M \times \Z$ such $a'$ is unmovable and $\alpha = a' \Delta^{k'}$, where $k'=k$ and $a' = \Phi^k(a)$.
\end{thm}

The expressions $\alpha = \Delta^k a$ and $\alpha = a' \Delta^{k'}$ given in Theorem \ref{thm2_2} are called the \emph{left $\Delta$-form} of $\alpha$ and the \emph{right $\Delta$-form} of $\alpha$, respectively.
Note that $a' \Delta^{k'} = \Phi^{-k}(\Delta^k a)$. 
We denote by $\sigma : \Div (\Delta) \to \Div (\Delta)$ the map such that $\Delta = u\, \sigma (u)$ for all $u \in \Div (\Delta)$.
Note that $\sigma (\Delta) = 1$ and $\sigma (1) = \Delta$.
The following is proved in Picantin \cite[Lemma 5.1]{Pican1}. 

\begin{prop}[Picantin \cite{Pican1}]\label{prop2_3}
Let $\alpha \in G$ and let $\alpha = v_q^{-1} \cdots v_2^{-1} v_1^{-1} u_1 u_2 \cdots u_p$ be its left greedy normal form.
\begin{itemize}
\item[(1)]
Suppose that $q=0$.
Let $r \in \{0,1, \dots, p\}$ such that $u_r = \Delta$ and $u_{r+1} \neq \Delta$.
Set $c = u_{r+1} \cdots u_p$.
Then the left $\Delta$-form of $\alpha$ is $\alpha = \Delta^r c$, and the left greedy normal form of $c$ is $c = u_{r+1} \cdots u_p$.
\item[(2)]
Suppose that $p=0$.
Let $r \in \{0,1, \dots, q\}$ such that $v_r = \Delta$ and $v_{r+1} \neq \Delta$.
For each $i \in \{r+1, \dots, q\}$ we set $w_i = \Phi^i (\sigma (v_i))$.
Let $c = w_q \cdots w_{r+1}$.
Then the left $\Delta$-form of $\alpha$ is $\alpha = \Delta^{-q} c$, and the left greedy normal form of $c$ is $c = w_q \cdots w_{r+1}$.
\item[(3)]
Suppose that $p \ge 1$ and $q \ge 1$.
For each $i \in \{1, \dots, q\}$ we set $w_i = \Phi^i (\sigma (v_i))$.
Let $c = w_q \cdots w_1 u_1 \cdots u_p$.
Then the left $\Delta$-form of $\alpha$ is $\alpha = \Delta^{-q} c$, and the left greedy normal form of $c$ is $c = w_q \cdots w_1 u_1 \cdots u_p$.
\end{itemize}
\end{prop}

Each of the cases of the proposition above  correspond to each of the cases of the corollary below.
\begin{corl}\label{corl2_4}
Let $\alpha \in G$ and let $\alpha = \Delta^p a$ be its left $\Delta$-form.
\begin{itemize}
\item[(1)]
If $p \ge 0$, then $\lg (\alpha) = \lg(a) + p$.
\item[(2)]
If $p \le -\lg(a)$, then $\lg (\alpha) = -p$.
\item[(3)]
If $- \lg(a) \le p \le 0$, then $\lg (\alpha) = \lg(a)$.
\end{itemize}
Summarizing $\lg (\alpha)= \max( \lg(a)+p, -p, \lg(a))$.
\end{corl}

The following theorem contains fundamental results on parabolic subgroups of Garside groups.

\begin{thm}[Godelle \cite{Godel1}]\label{thm2_5}
Let $(H, N, \delta)$ be a parabolic substructure of $(G, M, \Delta)$.
\begin{itemize}
\item[(1)]
$N= H \cap M$.
\item[(2)]
$H$ is a Garside group with Garside structure $(H, N, \delta)$.
\item[(3)]
Let $a \in N$, let $a = u_1 u_2 \cdots u_p$ be its left greedy normal form, and let $a = u_p' \cdots u_2' u_1'$ be its right greedy normal form with respect to $(G, M, \Delta)$.
Then $u_i,u_i' \in \Div (\delta)$ for all $i \in \{1, \dots, p\}$, $a = u_1 u_2 \cdots u_p$ is the left greedy normal form of $a$ with respect to $(H,N, \delta)$, and $a = u_p' \cdots u_2' u_1'$ is its right greedy normal form with respect to $(H, N, \delta)$.
\item[(4)]
Let $\alpha, \beta \in H$ and $\gamma \in G$.
If $\alpha \le_L \gamma \le_L \beta$, then $\gamma \in H$.
Similarly, if $\alpha \le_R \gamma \le_R \beta$, then $\gamma \in H$.
\item[(5)]
Let $\alpha, \beta \in H$.
Then $\alpha \wedge_L \beta, \alpha \vee_L \beta \in H$ and $\alpha \wedge_R \beta , \alpha \vee_R \beta \in H$.
\item[(6)]
Let $\alpha \in H$ and let $\alpha = b^{-1} a$ (resp. $\alpha = a' b'^{-1}$)  be its left orthogonal form (resp. be its right orthogonal form) with respect to $(G, M, \Delta)$.
Then $a,b \in N$ (resp. $a',b' \in N$) and $\alpha = b^{-1} a$ (resp. $\alpha = a' b'^{-1}$) is the left orthogonal form (resp. the right orthogonal form) of $\alpha$ with respect to $(H, N, \delta)$.
\end{itemize}
\end{thm}

The following fact will be used frequently.
\begin{rem}
Let $(H, N, \delta)$ be a parabolic substructure of $(G, M, \Delta)$. 
Let $a,b\in M$ such that $ab\in N$, then $a,b\in N$. 
Indeed, since $ab\in N$, $1\le_L a\le_L ab$ and by the above $a\in H\cap M=N$.
Similarly, $b \le_R ab$ and $1\le_R b \le_R ab$ and therefore $b\in H\cap M=N$.
\end{rem}

%%%%%%%%%%

\section{Transversal}\label{sec3}

>From now on we fix a Garside group $G$ with Garside structure $(G, M, \Delta)$ and a parabolic substructure $(H, N, \delta)$.
Our goal in the present section is to define a set $T$, to show that $T$ is a transversal of $H$ in $G$, and to show that $\lg (\theta) = \lg (H \theta)$ for all $\theta \in T$.
In order to define the set $T$ we need the following.

\begin{lem}[Dehornoy \cite{Dehor2}]\label{lem3_1}
Let $a \in M$.
There exists a unique $b \in N$ such that $\{ c \in N \mid c \le_L a \} = \{ c \in N \mid c \le_L b \}$.
\end{lem}

The element $b$ of Lemma \ref{lem3_1} is called the \emph{$N$-tail} of $a$ and is denoted by $b = \tau_N (a) = \tau (a)$.
We say that $a$ is \emph{$N$-reduced} if $\tau_N (a) = 1$.
Note that $a$ is $N$-reduced if and only if $a \wedge_L b = 1$ for all $b \in N$, or, equivalently, if and only if $a \wedge_L \delta = 1$.
On the other hand we set $\omega = \delta^{-1} \Delta \in M$.
This element will play a key role in our study.
Recall that $\Phi : G \to G$, $\alpha \mapsto \Delta \alpha \Delta^{-1}$, denotes the conjugation by $\Delta$.
Furthermore, we denote by $\varphi: H \to H$, $\beta \mapsto \delta \beta \delta^{-1}$, the conjugation by $\delta$ in $H$.
Note that $\omega^{-1} \beta \omega = (\Phi^{-1} \circ \varphi)(\beta)$ for all $\beta \in H$.
In particular, $\omega^{-1} N \omega = \Phi^{-1} (N) \subset M$.

Let $\alpha \in G$ and let $\alpha = a \Delta^p$ be its right $\Delta$-form. 
We say that $\alpha$ is \emph{$(H,N)$-reduced} if $a$ is $N$-reduced and either $p=0$ or ($p<0$ and $\omega \not \le_L a$).
Define $T$ to be the set of $(H,N)$-reduced elements of $G$.
The purpose of the present section is to prove the following two theorems.

\begin{thm}\label{thm3_3}
The set $T$ is a right transversal of $H$ in $G$.
\end{thm}

\begin{thm}\label{thm3_4}
We have $\lg (\theta) = \lg (H \theta)$ for all $\theta \in T$.
\end{thm}

The following three lemmas are preliminaries to the proofs of Theorem \ref{thm3_3} and Theorem \ref{thm3_4}.

\begin{lem}\label{lem3_5}
Let $b \in N$.
Then $b \wedge_L \omega = 1$ and $b \vee_L \omega = b \omega = \omega b'$, where $b' = (\Phi^{-1} \circ \varphi) (b)$.
\end{lem}

\begin{proof}
Let $x = b \wedge_L \omega$.
First observe that $1 \le_L x \le_L b$ and hence $x\in N$ and $\delta x\in N \subseteq M$.
Since $x \le_L \omega$, we have $\delta x \le_L \delta \omega = \Delta$, hence $\delta x \in \Div (\Delta)\cap N = \Div(\delta)$.
%On the other hand, $\delta x \le_L \delta b \in N$, hence $\delta x \in N$, since $N$ is a parabolic submonoid.
As $\delta x \in \Div(\delta)$ implies that $x^{-1}\in N$, we obtain that $x \in  N \cap N^{-1} = \{1\}$.
This finish the proof of the first claim.

Now, let $y = b \vee_L \omega$.
As previously observed, $b' = (\Phi^{-1} \circ \varphi) (b) = \omega^{-1} b \omega$, hence $b \omega = \omega b'$.
In particular, $b,\omega \le_L b \omega=\omega b'$ and hence $y\le_L b\omega= \omega b'$.
Let $y_1' \in M$ such that $\omega y_1' = y$.
Since $y \le_L \omega b'$, we have $y_1' \le_L b'$, hence there exists $y_2' \in M$ such that $y_1' y_2' = b'$.
We have $b' \in \Phi^{-1} (N)$ and $\Phi^{-1} (N)$ is a parabolic submonoid, hence $y_1', y_2' \in \Phi^{-1}(N)$.
Let $y_1 = (\varphi^{-1} \circ \Phi) (y_1')$ and $y_2 = (\varphi^{-1} \circ \Phi)(y_2')$.
Then $y_1, y_2 \in N$ and $b = y_1 y_2$.
We have $b = y_1 y_2 \le_L y = \omega y_1' = y_1 \omega$, hence $y_2 \le_L \omega$.
Since $y_2 \in N$ and, by the above, $y_2 \wedge_L \omega = 1$, it follows that $y_2 = 1$, hence $y_1 = b$, and therefore $y = b \omega$.
\end{proof}

The following is a direct consequence of Dehornoy--Paris \cite[Lemma 8.6, Lemma 8.7]{DehPar1}.

\begin{lem}\label{lem3_6}
Let $a,b_1,b_2 \in M$.
Then $\lg (b_1 a b_2) \ge \lg (a)$.
\end{lem}

\begin{lem}\label{lem3_7}
Let $c$ be a $N$-reduced element in $M$, let $k \ge 0$ be an integer, and let $d = \omega_1 \omega_2 \cdots \omega_k\, \Phi^{-k} (c)$, where $\omega_i = \Phi^{-i+1} (\omega)$ for all $i \in \{1, \dots, k\}$.
Then $d$ is $N$-reduced.
\end{lem}

\begin{rem}
Let $k \ge 0$ and let $\omega_i = \Phi^{-i+1} (\omega)$ for all $i \in \{1, \dots, k\}$ as in the above statement.
Then $\omega_1 \cdots \omega_k = \delta^{-k} \Delta^k$.
This equality will be often used in the remainder of the paper.
\end{rem}

\begin{proof}
We argue by induction on $k$.
The case $k = 0$ being trivial, we can assume that $k \ge 1$ and that the inductive hypothesis holds. 
Note that, if $k \ge 1$, then  $\omega  \le_L d$.
Let $b \in N$ such that $b \le_L d$.
In particular, $b \vee_L \omega \leq d$.
By Lemma \ref{lem3_5}, $b \vee_L \omega = \omega \, (\Phi^{-1} \circ \varphi)(b) \le_L \omega_1 \omega_2 \cdots \omega_k \, \Phi^{-k} (c)$, hence $(\Phi^{-1} \circ \varphi) (b) \le_L \omega_2 \cdots \omega_k \, \Phi^{-k} (c)$, and therefore $\varphi(b) \le_L \omega_1 \cdots \omega_{k-1} \, \Phi^{-k+1} (c)$. 
Recall that $\varphi(N) = N$.
By induction it follows that $\varphi (b) = 1$, hence $b = 1$.
\end{proof}

\begin{proof}[Proof of Theorem \ref{thm3_3}]
We start by showing that each right-coset $C$ of $H$ in $G$ contains an element of $T$.
Suppose first that $C$ contains an element $a \in M$.
Let $c \in M$ such that $a = \tau(a) c$.
Then $c \in T$ and $C = Ha = Hc$.
Now, we assume that $C$ contains no element of $M$ and note that elements not in $M$ have right $\Delta$-form $a \Delta^{-p}$ with $p \ge 1$.
Let $\alpha \in C$ and let $\alpha = a \Delta^{-p}$ be the right  $\Delta$-form of $\alpha$.
We choose $\alpha$ so that $p$ is minimal.
%Note that $p \ge 1$, since $C$ contains no element of $M$.
Let $c \in M$ such that $a = \tau(a) c$.
Set $\theta = c \Delta^{-p}$.
Since $C = H \alpha = H \theta$, it suffices to show that $\theta \in T$.
The element $c$ is unmovable since $c \le_R a$ and $a$ is unmovable.
Furthermore, $c$ is $N$-reduced since $a = \tau(a) c$.
Suppose that $\omega \le_L c$.
Let $c' \in M$ such that $c = \omega c'$.
Then $\delta \theta = \delta \omega c' \Delta^{-p} = \Phi(c') \, \Delta^{-p+1} \in H \theta = C$, which contradicts the minimality of $p$.
So, $\omega \not \le_L c$, thus $\theta = c \Delta^{-p} \in T$.

Now, we take two elements $\theta_1, \theta_2 \in T$, we suppose that there exists $\beta \in H$ such that $\beta \theta_1 = \theta_2$, and we prove that $\beta = 1$ and $\theta_1 = \theta_2$.
Let $\theta_1 = c_1 \Delta^{-p_1}$ and $\theta_2 = c_2 \Delta^{-p_2}$ be the right $\Delta$-forms of $\theta_1$ and $\theta_2$, respectively. 
By definition, $c_i$ is $N$-reduced, $p_i \ge 0$, and $\omega \not \le_L c_i$ if $p_i \ge 1$, for $i \in \{ 1, 2\}$.
We can assume without loss of generality that $p_1 \ge p_2$.
Let $\beta = b_2^{-1} b_1$ be the left orthogonal form of $\beta$.
Then $b_1 c_1 = b_2 c_2 \Delta^{p_1-p_2}$.
Suppose first that $p_1-p_2 > 0$.
In particular, $p_1 > 0$, hence $\omega \not \le_L c_1$.
Since $\Delta \le_R b_1 c_1$ in this case, we have $\Delta \le_L b_1 c_1$, hence $\omega \le_L b_1 c_1$, thus, by Lemma \ref{lem3_5}, $\omega \vee_L b_1 = b_1 \omega \le_L b_1 c_1$, and therefore $\omega \le_L c_1$: contradiction.
So, $p_1 = p_2$ and $b_1 c_1 = b_2 c_2$.
Let $y = b_1 \vee_L b_2$.
We have $y \in N$, since $b_1, b_2 \in N$ and $N$ is a parabolic submonoid.
Let $y_1 \in M$ such that $y = b_1 y_1$.
Here again, $y_1 \in N$, since $N$ is a parabolic submonoid.
We have $y = b_1 y_1 \le_L b_1 c_1$, since $b_1 c_1 = b_2 c_2$, hence $y_1 \le_L c_1$.
It follows that $y_1 = 1$, since $c_1$ is $N$-reduced, hence $y=b_1$.
Similarly, $y = b_2$, hence $b_1 = b_2 = 1$, because $b_1 \wedge_L b_2 = 1$.
So, $\beta=1$, and $\theta_1 = \theta_2$.
\end{proof}

Given $\theta\in T$ and $\beta\in H$ we will need to understand $\beta \theta$.
The following three lemmas are the cases for $\theta\in T,$ and $\beta \in H$ that will appear in the proofs of Theorems \ref{thm3_4} and \ref{thm5_5}. 

\begin{lem}\label{lem3_8}
Let $c \in T\cap M$ and $\beta \in H$. Let $\beta = b_2^{-1} b_1$ be the left orthogonal form of $\beta$. Then $1= b_2 \wedge_L (b_1c)$ and therefore  $b_2^{-1} (b_1 c)$ is the left orthogonal form of $\beta \theta$.
\end{lem}
\begin{proof}
Recall that $b_1, b_2 \in N$, since $H$ is a parabolic subgroup.
Set $y_1 = b_2 \wedge_L (b_1 c)$.
We have $y_1 \in N$ since $y_1 \le_L b_2$.
Let $y_2 = b_1 \vee_L y_1$.
We have $y_2 \in N$ since $y_1, b_1 \in N$.
Moreover, we have $y_2 \le_L b_1 c$ since $b_1 \le_L b_1 c$ and $y_1 \le_L b_1 c$. 
Let $y_3 \in M$ such that $y_2=b_1 y_3$.
We have $y_3 \in N$ (since $y_2 \in N$), $y_3 \le_L c$ and $c$ is $N$-reduced, hence $y_3 = 1$.
So, $y_2= b_1$, which implies that $y_1 \le_L b_1$.
We also have $y_1 \le_L b_2$ and $b_1 \wedge_L b_2 = 1$, hence $y_1=1$.
\end{proof}

\begin{lem}\label{lem3_9}
Let $\theta =c \Delta^{-p} \in T \setminus M$ in right $\Delta$-form with $p\geq 1$ and $\beta = b \in N$.
Then $bc$ is unmovable and $bc \Delta^{-p}$ is the right $\Delta$-form of $\beta \theta$.
\end{lem}
\begin{proof}
Notice that $c$ is unmovable and $N$-reduced, $\omega \not \le_L c$, and $p \ge 1$.
If we had $\Delta \le_L b c$, then we would have $\omega \le_L b c$, hence, by Lemma \ref{lem3_5}, we would have $b \vee_L \omega = b \omega \le_L b c$, and therefore we would have $\omega \le_L c$: contradiction.
So, $b c$ is unmovable. 
\end{proof}

\begin{lem}\label{lem3_10}
Let $\theta =c \Delta^{-p} \in T \setminus M$ in right $\Delta$-form with $p\geq 1$ and $\beta = b \delta^{-k} \in H \setminus N$ in right $\delta$-form with $k\geq 1$.
Then the right $\Delta$-form of $\beta \theta$ is $b\omega_1 \dots \omega_k\Phi^{-k}(c)\Delta^{-k-p}$, where $\omega_i = \Phi^{-i+1}$ for all $i \in \{1, \dots, k\}$.
\end{lem}
\begin{proof}
As $\theta\in T$, $c$ is unmovable and $N$-reduced, $\omega \not \le_L c$, and $p \ge 1$.
By using the equality $\delta^{-1} = \omega \Delta^{-1}$ it is easily seen that $\beta \theta = b \omega_1 \omega_2 \cdots \omega_k \, \Phi^{-k} (c) \Delta^{-p-k}$, where $\omega_i = \Phi^{-i+1} (\omega)$ for all $i \in \{1, \dots, k\}$.
By Lemma \ref{lem3_7} we have $\tau(b \omega_1 \cdots \omega_k \, \Phi^{-k} (c)) = b$.
Thus, if $\Delta \le_L b \omega_1 \cdots \omega_k \, \Phi^{-k} (c)$, then $\delta \le_L b \omega_1 \cdots \omega_k\, \Phi^{-k} (c)$, hence  $\delta \le_L b$: contradiction.
So, $\Delta \not \le_L b \omega_1 \cdots \omega_k \, \Phi^{-k} (c)$.
\end{proof}

\begin{proof}[Proof of Theorem \ref{thm3_4}]
We take $\theta \in T$ and $\beta \in H$ and we show that $\lg (\beta \theta) \ge \lg (\theta)$.
Suppose first that $\theta = c \in M$.
Let $\beta = b_2^{-1} b_1$ be the left orthogonal form of $\beta$.
%Recall that $b_1, b_2 \in N$, since $H$ is a parabolic subgroup.
%Set $y_1 = b_2 \wedge_L (b_1 c)$.
%We have $y_1 \in N$ since $y_1 \le_L b_2$.
%Let $y_2 = b_1 \vee_L y_1$.
%We have $y_2 \in N$ since $y_1, b_1 \in N$.
%Moreover, we have $y_2 \le_L b_1 c$ since $b_1 \le_L b_1 c$ and $y_1 \le_L b_1 c$. 
%Let $y_3 \in M$ such that $y_2=b_1 y_3$.
%We have $y_3 \in N$ (since $y_2 \in N$), $y_3 \le_L c$ and $c$ is $N$-reduced, hence $y_3 = 1$.
%So, $y_2= b_1$, which implies that $y_1 \le_L b_1$.
%We also have $y_1 \le_L b_2$ and $b_1 \wedge_L b_2 = 1$, hence $y_1=1$.
%This shows that $b_2^{-1} (b_1 c)$ is the left orthogonal form of $\beta \theta$.
Lemma \ref{lem3_8} imples that  $b_2^{-1} (b_1 c)$ is the left orthogonal form of $\beta \theta$.
By applying Theorem \ref{thm2_1} and Lemma \ref{lem3_6} we finally obtain $\lg (\beta \theta) = \lg (b_2) + \lg (b_1 c) \ge \lg (c) = \lg (\theta)$.

Assume that $\theta \not \in M$ and $\beta = b \in N$.
We write $\theta$ in the form $\theta = c \Delta^{-p}$, where $c$ is unmovable and $N$-reduced, $\omega \not \le_L c$, and $p \ge 1$.
%If we had $\Delta \le_L b c$, then we would have $\omega \le_L b c$, hence, by Lemma \ref{lem3_5}, we would have $b \vee_L \omega = b \omega \le_L b c$, and therefore we would have $\omega \le_L c$: contradiction.
%So, $b c$ is unmovable. 
By Lemma \ref{lem3_9} $bc$ is unmovable and $bc \Delta^{-p}$ is the right $\Delta$-form of $\beta \theta$. 
By Corollary \ref{corl2_4} it follows that $\lg (\beta \theta) = \max (\lg (bc), p) \ge \max(\lg (c), p) = \lg(\theta)$.

Assume that $\theta \not \in M$ and $\beta \not \in N$.
As before, we write $\theta$ in the form $\theta = c \Delta^{-p}$, where $c$ is unmovable and $N$-reduced, $\omega \not \le_L c$, and $p \ge 1$.
On the other hand we write $\beta$ in the form $\beta = b \delta^{-k}$, where $b \in N$, $\delta \not \le_L b$, and $k \ge 1$.
By Lemma \ref{lem3_10}, the right $\Delta$-form of $\beta \theta$ is  $b\omega_1 \cdots \omega_k\Phi^{-k}(c)\Delta^{-k-p}$, where $\omega_i = \Phi^{-i+1}$ for all $i \in \{1, \dots, k\}$.
%By using the equality $\delta^{-1} = \omega \Delta^{-1}$ it is easily seen that $\beta \theta = b \omega_1 \omega_2 \cdots \omega_k \, \Phi^{-k} (c) \Delta^{-p-k}$, where $\omega_i = \Phi^{-i+1} (\omega)$ for all $i \in \{1, \dots, k\}$.
%By Lemma \ref{lem3_7} we have $\tau(b \omega_1 \cdots \omega_k \, \Phi^{-k} (c)) = b$.
%Thus, if $\Delta \le_L b \omega_1 \cdots \omega_k \, \Phi^{-k} (c)$, then $\delta \le_L b \omega_1 \cdots \omega_k\, \Phi^{-k} (c)$, hence  $\delta \le_L b$: contradiction.
%So, $\Delta \not \le_L b \omega_1 \cdots \omega_k \, \Phi^{-k} (c)$.
By Corollary \ref{corl2_4} and Lemma \ref{lem3_6} it follows that $\lg (\beta \theta) = \max (\lg (b \omega_1 \cdots \omega_k \, \Phi^{-k} (c)), p+k) \ge \max (\lg (c), p) = \lg( \theta)$.
\end{proof}

%%%%%%%%%

\section{Regular language}\label{sec4}

A \emph{finite state automaton} is defined to be a quintuple $\FF = (X, \AA, \mu, Y, x_0)$, where $X$ is a finite set, called the \emph{set of states}, $\AA$ is a finite set, called the \emph{alphabet}, $\mu : X \times \AA \to X$ is a function, called the \emph{transition function}, $Y$ is a subset of $X$, called the \emph{set of accepted states}, and $x_0$ is an element of $X$, called the \emph{initial state}.
For $x \in X$ and $U = u_1 u_2 \cdots u_p \in \AA^*$ we define $\mu (x, U) \in X$ by induction on $p$ as follows. 
\[
\mu(x, U) =
\left\{ \begin{array}{ll} 
x &\text{if } p=0\,,\\
\mu (\mu (x,u_1 \cdots u_{p-1}), u_p) &\text{if } p \ge 1\,.
\end{array} \right.
\]
Then the set $L_{\FF} = \{ U \in \AA^* \mid \mu (x_0, U) \in Y\}$ is called the \emph{language recognized} by $\FF$.
A \emph{regular language} is a language recognized by a finite state automaton. 

Recall that $G$ is a given Garside group with Garside structure $(G, M, \Delta)$, that $(H, N, \delta)$ is a parabolic substructure, and that $T$ is the transversal of $H$ in $G$ defined in Section \ref{sec3}.
The following is used to understand the language defined by the automaton below.
Recall that $\sigma : \Div (\Delta) \to \Div (\Delta)$ is the function such that $\Delta = u \, \sigma(u)$ for all $u \in \Div (\Delta)$.

\begin{thm}[Dehornoy--Paris \cite{DehPar1}]\label{thm4_1}
Let $u_1, \dots, u_p, v_1, \dots, v_q \in \SS$.
Then $v_q^{-1} \cdots 
\allowbreak
v_2^{-1} v_1^{-1} u_1 u_2 \cdots u_p$ is a left greedy normal form if and only if $\sigma(u_i) \wedge_L u_{i+1} = 1$ for all $i \in \{1, \dots, p-1\}$, $\sigma (v_j) \wedge_L v_{j+1} = 1$ for all $j \in \{1, \dots, q-1\}$, and $u_1 \wedge_L v_1 = 1$.
\end{thm}

We define a finite state automaton $\FF = (X, \AA, \mu, Y, x_0)$ as follows.
We set $\AA = \SS \cup \SS^{-1}$, $X = \AA \cup \{x_0, x_1\}$, $Y = \AA \cup \{x_0\}$, and we define $\mu : X \times \AA \to X$ as follows. 
Let $u,v \in \SS$.
\begin{gather*}
\mu(x_0, v) = \left\{ \begin{array}{ll}
v &\text{if } v \wedge_L \delta = 1\,,\\
x_1 &\text{otherwise}\,,
\end{array}\right.\
\mu (x_0, v^{-1}) =\left\{ \begin{array}{ll}
v^{-1} &\text{if } v = \Delta\,,\\
v^{-1} &\text{if } \sigma(v) \wedge_L \delta = 1\\
&\ \text{ and } \omega \not \le_L \sigma(v)\,,\\
x_1 &\text{otherwise}\,,
\end{array}\right.\\
\mu(u, v) = \left\{ \begin{array}{ll}
v &\text{if } \sigma(u) \wedge_L v = 1\,,\\
x_1 &\text{otherwise}\,,
\end{array} \right.\
\mu (u, v^{-1}) = x_1\,,\\
\mu (u^{-1}, v) = \left\{ \begin{array}{ll}
v &\text{if } u \wedge_L v = 1\,,\\
x_1 &\text{otherwise}\,,
\end{array} \right.\
\mu (u^{-1}, v^{-1}) = \left\{ \begin{array}{ll}
v^{-1} &\text{if } \sigma(v) \wedge_L u = 1\,,\\
x_1 &\text{otherwise}\,,
\end{array} \right.\\
\mu(x_1 , v) = \mu (x_1 , v^{-1}) = x_1\,.
\end{gather*}

Recall that $\ev : \AA^* \to G$ is the map that sends each word $U$ to the element of $G$ that it represents. 
The purpose of the present section is to prove the following. 

\begin{thm}\label{thm4_2}
We have $\ev (L_\FF) = T$, the restriction $\ev : L_\FF \to T$ is a bijective  correspondence, and $\lg (U) = \lg (\ev (U))$ for all $U \in L_\FF$.
\end{thm}

Recall that for each $n \in \N$ we denote by $e(n) = e_{G,H,\SS}(n)$ the number of right cosets of $H$ in $G$ of length $n$.
Recall also that the \emph{coset growth series} of $H$ in $G$ is $\Gr_{G,H,\SS} (t) = \sum_{n=0}^\infty e(n)t^n$.
By Theorem \ref{thm3_3}, Theorem \ref{thm3_4} and Theorem \ref{thm4_2} $e(n)$ is equal to the number of words of length $n$ in the regular language $L_\FF$.
In other words, $\Gr_{G,H,\SS}(t)$ is the growth series of $L_\FF$.
Since we know that the growth series of a regular language is rational (see Flajolet--Sedgewick \cite{FlaSed1}, for example), it follows that:

\begin{corl}\label{corl4_3}
The formal series $\Gr_{G,H,\SS}(t)$ is rational.
\end{corl}

\begin{proof}[Proof of Theorem \ref{thm4_2}]
Set $T' = \ev (L_\FF)$.
By Theorem \ref{thm4_1} the restriction $\ev : L_\FF \to T'$ is a bijective correspondence, and $T'$ is the set of elements $\theta = v_q^{-1} \cdots v_2^{-1} v_1^{-1} u_1 u_2
\allowbreak
\cdots u_p$ written in left greedy normal form such that:
\begin{itemize}
\item
if $q = 0$ and $p \ge 1$, then $u_1 \wedge_L \delta = 1$,
\item
if $q \ge 1$, then either $v_q = \Delta$ or ($\sigma(v_q) \wedge_L \delta = 1$ and $\omega \not \le_L \sigma(v_q)$).
\end{itemize}
Moreover, by Theorem \ref{thm2_1}, we have $\lg (U) = \lg (\ev (U))$ for all $U \in L_\FF$.
So, it remains to show that $T' = T$.

Let $\theta = v_q^{-1} \cdots v_2^{-1} v_1^{-1} u_1 u_2 \cdots u_p$ be an element of $G$ written in left greedy normal form.
Suppose first that $q = 0$.
Then $\theta = u_1 u_2 \cdots u_p$.
If $p = 0$, then $\theta = 1 \in T \cap T'$.
So, we can assume that $p \ge 1$.
Since $\delta \in \Div(\Delta)$, we have $\theta \wedge_L \delta = (\theta \wedge_L \Delta) \wedge_L \delta = u_1 \wedge_L \delta$, hence 
\[
\theta \in T \ \Leftrightarrow\ \theta \wedge_L \delta = 1 \ \Leftrightarrow\ u_1 \wedge_L \delta = 1 \ \Leftrightarrow\ \theta \in T'\,.
\]

Suppose that $p = 0$ and $q \ge 1$.
Then $\theta = v_q^{-1} \cdots v_2^{-1} v_1^{-1}$.
Let $r \in \{0,1, \dots , q\}$ such that $v_i = \Delta$ for all $i \in \{1, \dots, r\}$ and $v_{r+1} \neq \Delta$.
Set $w_j = \Phi^{-j+1} (\sigma (v_{q-j+1}))$ for $j \in \{1, \dots, q-r\}$ and $c=w_1 \cdots w_{q-r}$.
Then, by Proposition \ref{prop2_3} and Theorem \ref{thm2_2}, $\theta = c \Delta^{-q}$, $c$ is unmovable, and $c = w_1 \cdots w_{q-r}$ is the left greedy normal form of $c$.
Notice that $w_1 = \sigma (v_q)$.
If $v_q = \Delta$, then $c = 1$ and $\theta = \Delta^{-q} \in T \cap T'$.
So, we can assume that $v_q \neq \Delta$, that is, $q > r$ and $c \neq 1$.
As above, since $\delta \in \Div (\Delta)$, we have $c \wedge_L \delta = w_1 \wedge_L \delta  = \sigma (v_q) \wedge_L \delta$.
On the other hand, since $\omega \in \Div (\Delta)$, we have $\omega \le_L c$ if and only if $\omega \le_L w_1 = \sigma (v_q)$.
So, 
\[
\theta \in T \ \Leftrightarrow\ (c \wedge_L \delta = 1 \text{ and } \omega \not \le_L c) \ \Leftrightarrow\ (\sigma (v_q) \wedge_L \delta = 1 \text{ and } \omega \not \le_L \sigma (v_q)) \ \Leftrightarrow\ \theta \in T'\,.
\]

Suppose that $p \ge 1$ and $q \ge 1$.
Then $\theta = v_q^{-1}  \cdots v_2^{-1} v_1^{-1} u_1 u_2 \cdots u_p$.
Since $u_1 u_2 \cdots u_p \wedge_L v_1 v_2 \cdots v_q = 1$, we have $u_i \neq \Delta$ for all $i \in \{1, \dots, p\}$ and $v_j \neq \Delta$ for all $j \in \{1, \dots, q\}$.
Set $w_j = \Phi^{-j+1} (\sigma (v_{q-j+1}))$ for all $j \in \{1, \dots, q\}$ and $w_j = \Phi^{-q} (u_{j-q})$ for all $j \in \{q+1, \dots, q+p\}$.
Let $c = w_1 \cdots w_{q+p}$.
Then, by Proposition \ref{prop2_3} and Theorem \ref{thm2_2}, $\theta = c \Delta^{-q}$, $c$ is unmovable, and $c = w_1 \cdots w_{q+p}$ is the left greedy normal form of $c$.
By applying the same reasoning as in the previous case, we then obtain 
\[
\theta \in T \ \Leftrightarrow\ (c \wedge_L \delta = 1 \text{ and } \omega \not \le_L c) \ \Leftrightarrow\ (\sigma(v_q) \wedge_L \delta = 1 \text{ and } \omega \not \le_L \sigma(v_q)) \ \Leftrightarrow\ \theta \in T'\,.
\]
\end{proof}

%%%%%%%%%%

\section{Projections}\label{sec5}

Let $G$ be a group generated by a finite set $\SS$.
We set $\AA = \SS \cup \SS^{-1}$ and, as before, we denote by $\ev : \AA^* \to G$ the map which sends each word $U \in \AA^*$ to the element of $G$ that it represents. 
We denote by $\lg = \lg_\SS$ the word length in $G$ with respect to $\SS$ and by $d : G \times G \to \N$ the distance function induced by $\lg$.
Recall that $d(\alpha, \beta) = \lg (\alpha^{-1} \beta)$ for all $\alpha, \beta \in G$. 
Recall also that, for $\alpha \in G$ and $X \subset G$, the \emph{distance} from $\alpha$ to $X$ is $d (\alpha, X) = \min \{d (\alpha, \beta) \mid \beta \in X \}$. 
The \emph{diameter} of a subset $X \subset G$ is $\diam (X) = \max \{ d(\alpha, \beta) \mid \alpha, \beta \in X \}$.

A word $U \in \AA^*$ is called a \emph{geodesic} (or a \emph{reduced word}) if $\lg (U) = \lg (\ev (U))$.
For a word $U = x_1 x_2 \cdots x_\ell$ of length $\ell$ and $i \in \N$ we set $U(i) = x_1 \cdots x_i$ if $i \le \ell$ and $U(i) = U$ if $i >\ell$.
Let $K$ be a positive constant.
We say that two words $U,V \in \AA^*$ \emph{$K$-fellow travel} if $d(\ev(U(i)), \ev(V(i))) \le K$ for all $i \in \N$.
We say that $(G,\SS)$ has the \emph{falsification by $K$-fellow traveller property} if for each non-geodesic word $U \in \AA^*$ there exists a strictly shorter word $V \in \AA^*$ such that $\ev (U) = \ev (V)$ and $U,V$ $K$-fellow travel.

Let $H$ be a subgroup of $G$.
The \emph{projection} of an element $\alpha \in G$ in $H$ is defined to be $\pi_H (\alpha) = \{ \beta \in H \mid d(\alpha, \beta) = d(\alpha, H) \}$.
Let $K$ be a positive constant. 
We say that $(G, \SS)$ has \emph{$K$-fellow projections} on $H$ if, for each $\alpha_1, \alpha_2 \in G$ such that $d(\alpha_1, \alpha_2) = 1$ and each $\beta_1 \in \pi_H (\alpha_1)$, there exists $\beta_2 \in \pi_H (\alpha_2)$ such that $d(\beta_1, \beta_2) \le K$.
We say that $(G,\SS)$ has \emph{$K$-bounded projections} on $H$ if, for each $\alpha_1, \alpha_2 \in G$ such that $d(\alpha_1, \alpha_2) = 1$, we have $\diam (\pi_H(\alpha_1) \cup \pi_H (\alpha_2)) \le K$. 
Note that, if $G$ has $K$-bounded projections on $H$, then $G$ has $K$-fellow projections on $H$.

As pointed out in the introduction, the starting point of our study was the following two theorems. 

\begin{thm}[Antol\'in \cite{Antol1}]\label{thm5_1}
Let $G$ be a group generated by a finite set $\SS$, let $\AA = \SS \cup \SS^{-1}$, and let $H$ be a subgroup of $G$.
Assume that $(G,\SS)$ has the falsification by fellow traveller property and has fellow projections on $H$.
Then:
\begin{itemize}
\item[(1)]
The language $\Geo (H \backslash G, \SS) = \{ U \in \AA^* \mid \lg (U) = \lg(H\, \ev(U)) \}$ is regular.
\item[(2)]
If, furthermore, $(G,\SS)$ has bounded projections on $H$, then the coset growth series $\Gr_{G, H, \SS}(t)$ of $H$ in $G$ is rational.
\end{itemize}
\end{thm}

\begin{thm}[Holt \cite{Holt1}]\label{thm5_2}
Let $G$ be a Garside group with Garside structure $(G, M, \Delta)$, and let $\SS = \Div (\Delta) \setminus \{1\}$. 
Then $G$ endowed with the generating set $\SS$ has the falsification by fellow traveller property.
\end{thm}

Let $G$ be a Garside group with Garside structure $(G, M, \Delta)$ and let $(H, N, \delta)$ be a parabolic substructure.
So, it is natural to ask whether $(G,\SS)$ has bounded projections on $H$ and, if not, whether it has fellow projections.
The answer is ``NO'' for the first question (see Corollary \ref{corl5_4}), and is ``YES'' for the second question (see Theorem  \ref{thm5_5}).
As pointed out in the introduction, this is the first example we know of a triple $(G,\SS,H)$, where $G$ is a group, $\SS$ a finite generating set and $H$ is a subgroup of $G$ such that $(G,\SS)$ has the falsification by fellow traveller property, it has fellow projections on $H$, but it does not have bounded projections on $H$.

>From now on $G$ denotes a given Garside group with Garside structure $(G, M, \Delta)$ and $(H, N, \delta)$ denotes a given parabolic substructure. 
Recall that $\omega = \delta^{-1} \Delta$, and $\Phi : G \to G$, $\alpha \mapsto \Delta \alpha \Delta^{-1}$, denotes the conjugation by $\Delta$.
The fact that $G$ does not have bounded projections on $H$ (if $H \neq G$) is a direct consequence of the following. 

\begin{lem}\label{lem5_3}
Suppose that $ H \neq G$.
Let $k \ge 1$ and let $d_k = \omega_1 \omega_2 \cdots \omega_k$, where $\omega_i = \Phi^{-i+1} (\omega)$ for all $i \in \{1, \dots, k\}$. 
Then $\diam (\pi_H(d_k)) \ge k$.
\end{lem}

\begin{proof}
By Lemma \ref{lem3_7} the element $d_k$ is $N$-reduced, hence $d_k \in T$, and therefore, by Theorem \ref{thm3_4}, $\lg (H d_k) = \lg (d_k)$. 
It follows that $d(d_k, H) = \lg(d_k)$, $1 \in \pi_H(d_k)$, and $\beta \in \pi_H (d_k)$ if and only if $\lg (\beta^{-1} d_k) = \lg(d_k)$ and $\beta \in H$.
We have $\delta \Delta = \Delta \, \sigma (\omega)$, hence $\sigma (\omega) = \Phi^{-1} (\delta)$, and therefore $\sigma (\omega_i) = \Phi^{-i} (\delta)$ for all $i \in \{1, \dots, k\}$.
Let $i \in \{1, \dots, k-1 \}$. 
Then $\omega_{i+1} = \Phi^{-i} (\omega)$ and, by Lemma \ref{lem3_5}, $\delta \wedge_L \omega = 1$, hence $\sigma (\omega_i) \wedge_L \omega_{i+1} = \Phi^{-i} (\delta) \wedge_L \Phi^{-i} (\omega) = 1$. 
It follows by Theorem \ref{thm4_1} that $\omega_1 \omega_2 \cdots \omega_k$ is the greedy normal form of $d_k$, hence $\lg (d_k) = k$.
On the other hand, $\delta^k d_k = \Delta^k$, hence, by Corollary \ref{corl2_4}, $\lg (\Delta^k) = k$. 
Since $\delta^k \in H$, this shows that $\delta^{-k} \in \pi_H (d_k)$. 
%Finally, by applying Corollary \ref{corl2_4} to $H$ endowed with the Garside structure $(H,N, \delta)$, we get $d(1, \delta^{-k}) = \lg (\delta^k) = k$, thus $\diam (\pi_H (d_k)) \ge k$.
Finally, by Corollary \ref{corl2_4}, $\delta^k$ has length $k$ with respect to $\Div(\delta) \setminus \{1 \}$, hence, by Theorem \ref{thm2_5}, $\lg (\delta^{-k}) = \lg (\delta^k) = k$.
So, $\diam (\pi_H (d_k)) \ge k$.
\end{proof}

\begin{corl}\label{corl5_4}
Suppose that $ H \neq G$ (and $H \neq 1$).
Then $G$ does not have bounded projections on $H$.
\end{corl}

\begin{proof}
For each $k \ge 1$ we choose $\alpha_k \in G$ such that $d(d_k, \alpha_k) = 1$. 
Then, by Lemma \ref{lem5_3}, $\diam (\pi_H (d_k) \cup \pi_H (\alpha_k)) \ge \diam (\pi_H (d_k)) \ge k$, hence there is no $K>0$ such that $\diam (\pi_H (d_k) \cup \pi_H (\alpha_k)) \le K$  for all $k \in \N$.
\end{proof}

The rest of the section is dedicated to the proof of the following.

\begin{thm}\label{thm5_5}
The group $G$ has $5$-fellow projections on $H$ with respect to $\SS$.
\end{thm}

The following three lemmas are preliminaries to the proof of Theorem \ref{thm5_5}.
Lemma \ref{lem5_6} is known to experts but to our knowledge is nowhere explicitly written in the literature.

\begin{lem}\label{lem5_6}
Let $a,b \in M$ and let $a = u_m \cdots u_2 u_1$ and $b = v_n \cdots v_2 v_1$ be the right greedy normal forms of $a$ and $b$, respectively.
If $a \le_R b$, then $m \le n$ and for all $i \in \{1,\dots, m\}$, one has that $u_i u_{i-1} \cdots u_1 \le_R v_i v_{i-1} \cdots v_1$.
\end{lem}

\begin{proof}
We already know by Theorem \ref{thm2_1} and Lemma \ref{lem3_6} that $m = \lg (a) \le \lg (b) = n$.
It remains to show that $u_i \cdots u_1 \le_R v_i \cdots v_1$ for all $i \in \{1, \dots, m\}$.
We argue by induction on $i$.
Note that for $x, y \in M$, if $x \le_R y$ then $(x \wedge_R \Delta) \leq_R (y \wedge_R \Delta)$.
Therefore, the case $i = 1$ is true by definition of a right greedy normal form. 
So, we can assume that $i \ge 2$ and that the inductive hypothesis holds.
By induction we have  $u_{i-1} \cdots u_1 \le_R v_{i-1} \cdots v_1$.
Let $c \in M$ such that $c u_{i-1} \cdots u_1 = v_{i-1} \cdots v_1$.
We have $u_i \le_R \Delta$, $u_m \cdots u_i \le_R v_n \cdots v_{i+1} v_i c$ and, by Dehornoy \cite[Lemme 3.10]{Dehor1}, $(v_n \cdots v_{i+1} v_i c) \wedge_R \Delta = (v_i c) \wedge_R \Delta$, hence $u_i \le_R v_i c$, and therefore $u_i u_{i-1} \cdots u_1 \le_R v_i v_{i-1} \cdots v_1$.
\end{proof}

\begin{lem}\label{lem5_7}
Let $a$ be an element of $M$ which can be written in the form $a = \omega_1 \omega_2 \cdots \omega_k c$, where $\omega_i = \Phi^{-i+1} (\omega)$ for all $i \in \{1, \dots, k \}$, $\Phi^{-k} (\omega) \not\le_L c$, and $c$ is $\Phi^{-k} (N)$-reduced. 
Then $\lg (a) = \lg(c) + k$.
\end{lem}

\begin{proof}
We argue by induction on $k$.
The case $k = 0$ being trivial, we can assume that $k \ge 1$ and that the inductive hypothesis holds.
It suffices to show that $a \wedge_L \Delta = \omega = \omega_1$. 
Indeed, in that case, by Theorem \ref{thm2_1}, we have $\lg (a) = \lg(\omega_2 \cdots \omega_k c) + 1$. Since $\Phi( \SS ) = \SS$, $\Phi$ preserves $\lg$, thus $\lg (\omega_2 \cdots \omega_k c) = \lg (\omega_1 \omega_2 \cdots  \omega_{k-1} \Phi(c))$ and $\Phi(c)$ is $\Phi^{-k+1}(N)$-reduced. By induction, $\lg (\omega_2 \cdots \omega_k c) = \lg(c) + k -1$, hence $\lg (a) = \lg (c) + k$. 

Set $x = a \wedge_L \Delta$. 
Since $\omega = \omega_1 \le_L a$ and $\omega \le_L \Delta$, we have $\omega \le_L x$. 
Let $x_1 \in M$ such that $x = \omega x_1$. 
We have $\omega x_1\, \sigma (x) = x\, \sigma(x) = \Delta = \omega \delta_1$, where $\delta_1 = \Phi^{-1} (\delta)$, hence $x_1 \le_L \delta_1$, thus $x_1 \in \Phi^{-1}(N)$, since $\Phi^{-1} (N)$ is a parabolic submonoid of $M$.
We also have $x_1 \le_L \omega_2 \cdots \omega_k c$ and $\omega_2 \cdots \omega_k c$ is $\Phi^{-1} (N)$-reduced by Lemma \ref{lem3_7}, hence $x_1 = 1$ and $x = \omega$. 
\end{proof}

\begin{lem}\label{lem5_8}
Let $\alpha \in G$, $\ell = \lg (H \alpha)$ and $\Min (\alpha) = \{ \gamma \in H \alpha \mid \lg(\gamma) = \ell \}$. 
Then $\ell = d (\alpha, H)$ and we have a bijective map $\Min (\alpha) \to \pi_H (\alpha)$ which sends $\gamma$ to $\alpha \gamma^{-1}$ for all $\gamma \in \Min (\alpha)$.
\end{lem}

\begin{proof}
If $\beta \in H$, then $d(\beta, \alpha) = \lg (\beta^{-1} \alpha) \ge \ell$. 
On the other hand, if $\gamma \in \Min (\alpha)$ and $\beta = \alpha \gamma^{-1}$, then $\beta \in H$ and $d (\beta, \alpha) = \lg(\beta^{-1} \alpha) = \lg (\gamma) = \ell$.
So, $\ell = d(\alpha, H)$ and $\beta \in \pi_H(\alpha)$.
Reciprocally, if $\beta \in \pi_H(\alpha)$ and $\gamma = \beta^{-1} \alpha \in H \alpha$, then $\lg (\gamma) = d(\beta, \alpha) = \ell$ and $\gamma \in \Min (\alpha)$.
\end{proof}

The proof of Theorem \ref{thm5_5} is divided into three cases, each case being treated in one of the following three lemmas.

\begin{lem}\label{lem5_9}
Let $\alpha \in G$, $\beta \in \pi_H (\alpha)$ and $u \in \SS$.
Denote by $\theta$ the element of $T$ such that $H \alpha = H \theta$ and assume that $\theta = c \in M$. 
Then there exists $\beta' \in \pi_H(\alpha u)$ such that $d(\beta, \beta') \le 3$.
\end{lem}

\begin{proof}
Set $\ell = \lg (H \alpha)$. 
By Theorem \ref{thm3_4}, $\lg(\theta) = \ell$ and $c=\theta \in \Min (\alpha)$.
By Lemma \ref{lem5_8} there exists $\gamma \in \Min (\alpha)$ such that $\beta = \alpha \gamma^{-1}$. 
Let $\beta_1 \in H$ such that $\beta_1 \theta = \gamma$. 
Let $b_1'^{-1} b_1$ be the left orthogonal form of $\beta_1$. 
By Lemma \ref{lem3_8} we have that $b_1' \wedge_L (b_1c) = 1$.
Thus, by Theorem \ref{thm2_1}, $\ell = \lg(\gamma) = \lg(b_1') + \lg (b_1c)$ and, by Lemma \ref{lem3_6}, $\lg(b_1 c) \ge \lg(c)$.
Since $\ell = \lg (c )$, it follows that $b_1' = 1$, $\beta_1 = b_1 \in N$ and $\gamma = b_1 c \in M$.

Set $b_2 = \tau_N (cu) \in N$ and denote by $c'$ the element of $M$ such that $ cu = b_2 c'$. 
Note that $c'$ is $N$-reduced, hence $c' \in T$.
Note also that $H c' = H \alpha u$, hence $\theta' = c'$ is the element of $T$ which lies in $H \alpha u$, and, by Theorem \ref{thm3_4}, $\lg (H \alpha u) = \lg(c')$.
Since $c$ is $N$-reduced, we have $b_2 \wedge_L c = 1$, hence, by Theorem \ref{thm2_1}, $\lg (b_2^{-1} c) = \lg (b_2) + \lg(c)$.
So, 
\[
\lg(c) + 1 \ge \lg (cu) =\lg(b_2c') \ge \lg(c') = \lg(b_2^{-1} cu) \ge \lg(b_2^{-1} c) -1 = \lg(b_2) + \lg(c) -1\,,
\]
hence $\lg(b_2) \le 2$.
It follows that  
\[
\lg(c') \ge \lg(b_2 c') - \lg(b_2) \ge \lg(cu) -2 \ge \lg(c)-2\,.
\]
Now we show that there exist $b_3, b_4 \in N$ such that $b_3 b_4 = b_1 b_2$, $\lg(b_4 c') = \lg(c')$ and $\lg(b_3) \le 3$.
Set $n = \lg (b_1 b_2 c')$ and $m= \lg(c')$.
We have  
\[
n = \lg(b_1 b_2 c') = \lg (b_1 c u) \le \lg(b_1 c) + 1 = \lg(c) + 1 \le \lg(c') + 3 = m+3\,.
\]
Let $b_1 b_2 c' = v_n \cdots v_2 v_1$ be the right greedy normal form of $b_1 b_2 c'$.
Then $c' \le_R v_n \cdots v_2 v_1$. 
Set $d = v_m \cdots v_2 v_1$. 
Since $\lg( c' )=m$, by Lemma \ref{lem5_6} we have $c' \le_R d$ and, by Theorem \ref{thm2_1}, $m = \lg (d)$. 
Let $b_4 \in M$ such that $b_4 c' = d$, and let $b_3 = v_n \cdots v_{m+1}$. 
We have $b_3 b_4 c' = b_1 b_2 c'$, hence $b_3 b_4 =  b_1 b_2$, therefore $b_3, b_4 \in N$ since $b_1 b_2 \in N$ and $N$ is a parabolic submonoid. 
By the above we also have $\lg(b_3) = n-m \le 3$ and $\lg(b_4 c') = \lg(d) = lg(c')$.

Set $\gamma' = b_4 c'$ and $\beta' = (\alpha u) \gamma'^{-1}$.
By the above, $\gamma' \in \Min (\alpha u)$, hence, by Lemma \ref{lem5_8}, $\beta' \in \pi_H (\alpha u)$. 
Moreover, 
\[
\beta^{-1} \beta' = \gamma u \gamma'^{-1} = b_1 c u c'^{-1} b_4^{-1} = b_1 b_2 b_4^{-1} = b_3\,,
\]
hence $d(\beta, \beta') = \lg(b_3) \le 3$.
\end{proof}

\begin{lem}\label{lem5_10}
Let $\alpha \in G$, $\beta \in \pi_H (\alpha)$ and $u \in \SS$.
Denote by $\theta$ the element of $T$ such that $H \theta = H \alpha$ and assume that $\theta \in M^{-1} \setminus \{1\}$.
Then there exists $\beta' \in \pi_H (\alpha u)$ such that $d(\beta, \beta') \le 3$.
\end{lem}

\begin{proof}
Let $\theta = c \Delta^{-t}$ be the right $\Delta$-form of $\theta$ and $\ell = \lg (\theta)$.
By Theorem \ref{thm3_4}, $\ell = \lg (H \alpha)$ and $\theta \in \Min (\alpha)$.
Since $\theta \in M^{-1}$ and $\theta \neq 1$, by Proposition \ref{prop2_3}~(2), $t\geq \lg (c)$  and Corollary \ref{corl2_4}, $\ell = t \ge \lg(c)$ and $\ell \ge 1$. 
Moreover, since $\theta \in T$, the element $c$ is $N$-reduced and $\omega \not \le_L c$. 
By Lemma \ref{lem5_8}, there exists $\gamma \in \Min (\alpha)$ such that $\beta = \alpha \gamma^{-1}$. 
Let $\beta_1 \in H$ such that $\beta_1 \theta = \gamma$. 
We show that $\beta_1 = b_1 \in N$. 
Suppose instead that $\beta_1 \not \in N$.
Then $\beta_1$ is written $\beta_1 = b_1 \delta^{-k}$ where $k \ge 1$ and $\delta \not \le_L b_1$. 
By Lemma \ref{lem3_10}, the right $\Delta$-form of $\gamma$ is $\gamma = b_1 \omega_1 \omega_2 \cdots \omega_k \Phi^{-k}(c) \Delta^{-k -\ell}$, where $\omega_i = \Phi^{-i+1} (\omega)$. 
By Corollary \ref{corl2_4} and Lemma \ref{lem5_7} it follows that  
\[
\ell = \lg(\gamma) = \max (\lg (b_1 \omega_1 \omega_2 \cdots \omega_k \Phi^{-k}(c)), k+ \ell)
\ge k + \ell > \ell\,:
\]
contradiction.
So, $\beta_1 = b_1 \in N$ and $\gamma = (b_1 c) \Delta^{- \ell}$.
By Lemma \ref{lem3_9}, $b_1 c$ is unmovable, hence $\gamma = (b_1 c)\Delta^{- \ell }$ is the right $\Delta$-form of $\gamma$ and, by Corollary \ref{corl2_4}, $\ell = \lg(\gamma) = \max( \lg(b_1c), \ell)$.
Thus, $\lg(c) \le \lg(b_1 c) \le \ell$.

We have $\theta u = c \Delta^{- \ell} u = c u_1 \Delta^{-\ell}$, where $u_1 = \Phi^{-\ell}(u) \in \SS$.
Let $b_2 = \tau (c u_1) \in N$ and let $c_1 \in M$ such that $c u_1 = b_2 c_1$.
We write $c_1$ in the form $c_1 = \omega_1 \omega_2 \cdots \omega_p c_2$, where $\omega_i = \Phi^{-i+1} (\omega)$ for all $i \in \{1, \dots, p\}$, and $\Phi^{-p} (\omega) \not \le_L c_2$.
We show that $p \le 1$.
Assume instead that $p \ge 2$.
Let $u_1' \in \Div (\Delta)$ such that $u_1' u_1 = \Delta$.
Recall that $\varphi : H \to H$ denotes the conjugation by $\delta$ in $H$.
Then 
\begin{gather*}
\omega_1 \omega_2 \le_L c_1 \ \Rightarrow\  
b_2 \omega_1 \omega_2 \le_L b_2 c_1 = c u_1 \ \Rightarrow\
\delta^2 b_2 \omega_1 \omega_2 = \varphi^2(b_2) \Delta^2 \le_L \delta^2 c u_1\,.
\end{gather*}
Observe that for $x\in M$ one has that $\Delta^i\le_L x \Delta^i$ and $\Delta^i \le_R \Delta^i x$. We have that
\begin{gather*}
\omega_1 \omega_2 \le_L c_1 \ \Rightarrow\
\Delta^2 \le_L \delta^2 c u_1 \ \Rightarrow\
\Delta^2 \le_R  \delta^2 c u_1 \ \Rightarrow\  
\Delta u_1' \le_R \delta^2 c \ \Rightarrow\
\Delta \le_R \delta^2 c \ \Rightarrow\\
\Delta \le_L \delta^2 c \ \Rightarrow\
\omega \le_L \delta^2 c \ \Rightarrow\
\omega \wedge_L \delta^2 = \delta^2 \omega \le_L \delta^2 c \ \Rightarrow\
\omega \le_L c\,:
\end{gather*}
contradiction.
So, $p \le 1$.

Let $b_1' = \omega_p^{-1} \cdots \omega_1^{-1} b_1 \omega_1 \cdots \omega_p \in \Phi^{-p}(N)$ and $b_2' = \omega_p^{-1} \cdots \omega_1^{-1} b_2 \omega_1 \cdots \omega_p \in \Phi^{-p}(N)$.
Then $b_1 c u_1 = b_1 b_2 c_1 = b_1 b_2 \omega_1 \cdots \omega_p c_2 = \omega_1 \cdots \omega_p b_1' b_2' c_2$.
We show that there exist $b_3', b_4' \in \Phi^{-p}(N)$ such that $b_3' b_4' = b_1' b_2'$, $\lg (b_4' c_2) \le \max(\ell-p, \lg (c_2))$ and $\lg (b_3') \le 2$.
Set  $n = \lg (b_1' b_2' c_2)$ and $m = \max (\ell-p, \lg (c_2))$.
We have  
\[
\lg (c_2) \le \lg (b_2 \omega_1 \cdots \omega_p c_2) = \lg (c u_1) \le \lg (c) + 1 \le \ell +1 \le \ell - p + 2\,,
\] 
hence $\ell - p \le m \le \ell - p + 2$.
So, 
\begin{gather*}
n - m \le n - \ell + p \le n - \ell + 1 \le \lg( \omega_1 \cdots \omega_p b_1' b_2' c_2) - \ell + 1\\
= \lg (b_1 b_2 \omega_1 \cdots \omega_p c_2) - \ell + 1
= \lg (b_1 c u_1) - \ell + 1 \le \lg (b_1 c) - \ell + 2 \le 2\,.
\end{gather*}
Let $b_1' b_2' c_2 = v_n \cdots v_2 v_1$ be the right greedy normal form of $b_1' b_2' c_2$.
Suppose first that $m \ge n$.
Set $b_3' = 1$ and $b_4' = b_1' b_2'$.
Then $\lg (b_4' c_2) = \lg (b_1' b_2' c_2) = n \le m = \max(\ell - p, \lg(c_2))$ and $\lg (b_3') = 0 \le 2$.
Suppose now that $m < n$.
Let $d = v_m \cdots v_2 v_1$.
Since $\lg( c_2) \leq m$, by Lemma \ref{lem5_6} we have that  $c_2 \le_R d$.
Let $b_4' \in M$ such that $b_4' c_2 = d$, and let $b_3' = v_n \cdots v_{m+1}$.
Then $\lg (b_3') \le n-m \le 2$, $\lg (b_4' c_2) \le m = \max(\ell - p, \lg (c_2))$, $b_3' b_4' = b_1' b_2'$, since $b_3'b_4' c_2 =
b_1' b_2' c_2$, and $b_3', b_4' \in \Phi^{-p}(N)$, since $b_3' b_4' = b_1' b_2' \in \Phi^{-p}(N)$ and $\Phi^{-p}(N)$ is a parabolic submonoid.

Let $c' = \Phi^{p} (c_2)$ and $\theta' = c' \Delta^{-\ell + p}$.
We have $\Phi^{-p} (\omega) \not \le_L c_2$, hence $\omega \not \le_L c'$.
Since $\omega \le_L \Delta$, this also implies that $\Delta \not \le_L c'$, that is, $c'$ is unmovable. 
Let $b' \in \Phi^{-p} (N)$.
Set $b = \omega_1 \cdots \omega_p b' \omega_p^{-1} \cdots \omega_1^{-1} \in N$.
If $b' \le_L c_2$, then $\omega_1 \cdots \omega_p b' = b \omega_1 \cdots \omega_p \le_L \omega_1 \cdots \omega_p c_2 = c_1$, hence  $b \le_L c_1$.
Since $c_1$ is $N$-reduced, it follows that $b = 1$, hence $b' = 1$.
This shows that $c_2$ is $\Phi^{-p} (N)$-reduced, hence $c'$ is $N$-reduced. 
Finally, since $\ell \ge 1$ and $p \le 1$, we have $\ell - p \ge 0$.
So, $\theta' \in T$.
On the other hand,
\begin{gather*}
\alpha u = 
\beta \gamma u = 
\beta b_1 c \Delta^{-\ell} u = 
\beta b_1 c u_1 \Delta^{-\ell} = 
\beta b_1 b_2 \omega_1 \cdots \omega_p c_2 \Delta^{-\ell}\\ = 
\beta b_1 b_2 \delta^{-p} \Delta^p c_2 \Delta^{-\ell} = 
\beta b_1 b_2 \delta^{-p} c' \Delta^{-\ell + p} = 
\beta b_1 b_2 \delta^{-p} \theta'\,,
\end{gather*}
hence $\theta'$ is the element of $T$ which lies in $H \alpha u$.

Set $b_3 = \Phi^{p} (b_3')$ and $b_4 = \Phi^{p} (b_4')$.
Then $b_3, b_4 \in N, \lg (b_3) \le 2$, and $\lg (b_4 c') = \lg (b_4' c_2) \le \max (\ell-p, \lg (c_2))  = \max (\ell-p, \lg(c')) = \lg (\theta')$.
Let $\gamma' = (b_4 c') \Delta^{-\ell+p} = b_4 \theta'$.
If we had $\Delta \le_L b_4 c'$, then we would have $\omega \le_L b_4 c'$, hence we would have $\omega \vee_L b_4 = b_4 \omega \le_L b_4 c'$, and therefore $\omega \le_L c'$: contradiction.
So, $b_4 c'$ is unmovable, hence $(b_4 c') \Delta^{-\ell+p}$ is the right $\Delta$-form of $\gamma'$.
By Corollary \ref{corl2_4} and the above it follows that $\lg (\gamma') = \max (\lg (b_4 c'),\ell-p) = \max (\lg(c'), \ell-p) = \lg (\theta')$, hence, since $\theta' \in \Min (\alpha u)$, we have $\gamma' \in \Min (\alpha u)$.
Let $\beta' =  \alpha u \gamma'^{-1}$.
We have $\beta' \in \pi_H (\alpha u)$ by Lemma \ref{lem5_8}.
Moreover, 
\begin{gather*}
\beta^{-1} \beta' 
= \gamma u \gamma'^{-1}
= b_1 c \Delta^{-\ell} u \Delta^{\ell-p} c'^{-1} b_4^{-1} 
= b_1 b_2 c_1 \Delta^{-p} c'^{-1} b_4^{-1}\\
= b_1 b_2 \omega_1  \cdots \omega_p c_2 \Delta^{-p} c'^{-1} b_4^{-1}
= \omega_1 \cdots \omega_p b_1' b_2' c_2 \Delta^{-p} c'^{-1} b_4^{-1}\\
= \delta^{-p} \Delta^{p} b_3' b_4' c_2 \Delta^{-p} c'^{-1} b_4^{-1} 
= \delta^{-p} b_3\,.
\end{gather*}
So, $d(\beta, \beta') = \lg (\delta^{-p} b_3) \le p + \lg (b_3) \le 3$.
\end{proof}

\begin{lem}\label{lem5_11}
Let $\alpha \in G$, $\beta \in \pi_H (\alpha)$ and $u \in \SS$.
Denote by $\theta$ the element of $T$ such that $H \theta = H \alpha$ and assume that $\theta \not \in M \cup M^{-1}$.
Then there exists $\beta' \in \pi_H (\alpha u)$ such that $d(\beta, \beta') \le 5$.
\end{lem}

\begin{proof}
Let $\theta = c \Delta^{-\ell}$ be the right $\Delta$-form of $\theta$.
Since $\theta \not \in M \cup M^{-1}$, by Proposition \ref{prop2_3} and Corollary \ref{corl2_4}, we have $1 \le \ell < \lg (c)$ and $\lg (\theta) = \lg(c)$.
Furthermore, by Theorem \ref{thm3_4}, $\lg (H \alpha) = \lg(\theta)$ and so $\theta \in \Min (\alpha)$.
Moreover, since $\theta \in T$, the element $c$ is $N$-reduced and $\omega \not \le_L c$.
By Lemma \ref{lem5_8} there exists $\gamma \in \Min (\alpha)$ such that $\beta = \alpha \gamma^{-1}$.
Let $\beta_1 \in H$ such that $\beta_1 \theta = \gamma$.
We show that $\beta_1 = b_1 \in N$.
Suppose instead that $\beta_1 \not \in N$.
Then $\beta_1$ is written $\beta_1 = b_1 \delta^{-k}$ where $k \ge 1$, $b_1 \in N$, and $\delta \not \le_L b_1$.
By Lemma \ref{lem3_10}, the right $\Delta$-form of $\gamma$ is $\gamma = b_1 \omega_1 \omega_2 \cdots \omega_k \Phi^{-k}(c) \Delta^{-k-\ell}$, where $\omega_i = \Phi^{-i+1} (\omega)$.
Then, by Corollary \ref{corl2_4} and Lemma \ref{lem5_7}, $\lg (\gamma) = \max (\lg (b_1 \omega_1 \cdots \omega_k \Phi^{-k} (c)), k+ \ell) \ge \max (\lg (\omega_1 \cdots \omega_k \Phi^{-k} (c)), k+ \ell) = \max (\lg (c) + k, k + \ell) > \lg (c) = \lg (\theta)$: contradiction.
So, $\beta_1 = b_1 \in N$ and $\gamma = (b_1 c) \Delta^{-\ell}$.
By Lemma \ref{lem3_9}, $b_1 c$ is unmovable, hence $\gamma = (b_1 c) \Delta^{-\ell}$ is the right $\Delta$-form of $\gamma$.
By Corollary \ref{corl2_4}, $\lg (c) = \lg (\theta) = \lg (\gamma) = \max (\lg (b_1 c),\ell) \ge \lg (b_1 c)$ and, by Lemma \ref{lem3_6}, $\lg (b_1 c) \ge \lg (c)$, hence $\lg (c) = \lg (b_1 c)$.

We have $\theta u = c \Delta^{-\ell} u = c u_1 \Delta^{-\ell}$, where $u_1 = \Phi^{-\ell}(u) \in \SS$.
Let $b_2 = \tau (c u_1) \in N$ and let $c_1 \in M$ such that $c u_1 = b_2 c_1$.
Since $c$ is $N$-reduced, we have $b_2 \wedge_L c = 1$, hence, by Theorem \ref{thm2_1}, $\lg (b_2^{-1} c) = \lg (b_2) + \lg (c)$.
It follows that 
\begin{gather*}
\lg (c) + 1 \ge \lg (c u_1) = \lg (b_2 c_1) \ge \lg(c_1) = \lg (b_2^{-1} c u_1)\\
\ge \lg (b_2^{-1} c) - 1 = \lg (b_2) + \lg(c) -1\,,
\end{gather*}
hence $\lg (b_2) \le 2$.

We write $c_1$ in the form $c_1 = \omega_1 \omega_2 \cdots \omega_p c_2$, where $\omega_i = \Phi^{-i+1} (\omega)$ for all $i \in \{1, \dots, p\}$, and $\Phi^{-p} (\omega) \not \le_L c_2$.
We show that $p \le 1$ in the same way as in the proof of Lemma \ref{lem5_10}.
Moreover, by the above,
\[
\lg (c_2) \ge \lg (b_2 \omega_1 \cdots \omega_p c_2) - \lg (b_2) - \lg (\omega_1 \cdots \omega_p) \ge \lg (c u_1) -3 \ge \lg (c) -3\,.
\]

Let $b_1' = \omega_p^{-1} \cdots \omega_1^{-1} b_1 \omega_1 \cdots \omega_p \in \Phi^{-p} (N)$ and $b_2' = \omega_p^{-1} \cdots \omega_1^{-1} b_2 \omega_1 \cdots \omega_p \in \Phi^{-p} (N)$.
Then $b_1 c u_1 = b_1 b_2 c_1 = b_1 b_2 \omega_1 \cdots \omega_p c_2 = \omega_1 \cdots \omega_p b_1' b_2' c_2$.
We show now that there exist $b_3', b_4' \in \Phi^{-p} (N)$ such that $b_3' b_4' = b_1' b_2'$, $\lg (b_4' c_2) = \lg (c_2)$ and $\lg (b_3') \le 4$.
Set $n = \lg (b_1' b_2' c_2)$ and $m = \lg (c_2)$.
We have $n \ge m$ by Lemma \ref{lem3_6} and, by the above, 
\begin{gather*}
n = \lg (b_1' b_2' c_2) \le \lg (\omega_1 \cdots \omega_p b_1' b_2' c_2) = \lg (b_1 c u_1) \le \lg (b_1 c) + 1\\
 = \lg(c) + 1 \le \lg (c_2)+4 = m+4\,.
\end{gather*}
Let $b_1' b_2' c_2 = v_n \cdots v_2 v_1$ be the right greedy normal form of $b_1' b_2' c_2$.
Set $d = v_m \cdots v_2 v_1$.
By Lemma \ref{lem5_6} we have $c_2 \le_R d$.
Let $b_4' \in M$ such that $b_4' c_2 = d$ and let $b_3' = v_n \cdots v_{m+1}$.
We have $b_3' b_4' c_2 = b_1' b_2' c_2$, hence $b_3' b_4' = b_1' b_2'$.
In particular, $b_3', b_4' \in \Phi^{-p}(N)$, since $b_1' b_2' \in \Phi^{-p} (N)$ and $\Phi^{-p} (N)$ is a parabolic submonoid.
Furthermore, by the above, $\lg (b_3') = n-m \le 4$ and $\lg (b_4' c_2) = \lg (d) = m = \lg(c_2)$.

Let $c' = \Phi^{p} (c_2)$ and let $\theta' = c' \Delta^{-\ell+p}$.
We show in the same way as in the proof of Lemma \ref{lem5_10} that $\theta' = c' \Delta^{-\ell+p}$ is the right $\Delta$-form of $\theta'$, that $ \omega \not \le_L c'$, that $\theta' \in T$, and that $\theta' \in H \alpha u$.
Set $b_3 = \Phi^p (b_3')$ and $b_4 = \Phi^p(b_4')$.
Then $b_3, b_4 \in N$, $\lg (b_3) \le 4$, and $\lg (b_4 c') = \lg (b_4' c_2) = \lg (c_2) = \lg (c')$.
Let $\gamma' = (b_4 c') \Delta^{-\ell+p} = b_4 \theta'$.
If we had $\Delta \le_L b_4 c'$, then we would have $\omega \le_L b_4 c'$, hence we would have $\omega \vee_L b_4 = b_4 \omega \le_L b_4 c'$, and therefore $\omega \le_L c'$: contradiction.
So, $b_4 c'$ is unmovable, hence $(b_4 c') \Delta^{-\ell+p}$ is the right $\Delta$-form of $\gamma'$.
By Corollary \ref{corl2_4} and the above it follows that $\lg (\gamma') = \max (\lg (b_4 c'), \ell-p) = \max (\lg (c'), \ell-p) =
\lg (\theta')$.
So, since $\theta' \in \Min (\alpha u)$, we have $\gamma' \in \Min (\alpha u)$.
Let $\beta' = \alpha u \gamma'^{-1}$.
We have $\beta' \in \pi_H (\alpha u)$ by Lemma \ref{lem5_8}.
Furthermore, 
\begin{gather*}
\beta^{-1} \beta' = \gamma u \gamma'^{-1}
= b_1 c \Delta^{-\ell} u \Delta^{\ell-p} c'^{-1} b_4^{-1} 
= b_1 b_2 \omega_1 \cdots \omega_p c_2 \Delta^{-p} c'^{-1} b_4^{-1}\\
= \omega_1 \cdots \omega_p b_1' b_2' c_2 \Delta^{-p} c'^{-1} b_4^{-1}
= \delta^{-p} \Delta^{p} b_3' b_4' c_2 \Delta^{-p} c'^{-1} b_4^{-1} 
= \delta^{-p} b_3\,.
\end{gather*}
Thus, $d(\beta,\beta') = \lg(\delta^{-p} b_3) \le p + \lg (b_3) \le 5$.
\end{proof}

\begin{proof}[Proof of Theorem \ref{thm5_5}]
Follows directly from Lemma \ref{lem5_9}, Lemma \ref{lem5_10} and Lemma \ref{lem5_11}.
\end{proof}

%%%%%%%%%%
\subsection*{Acknowledgments}
The first author acknowledges partial
support from the Spanish Government through grant number MTM2017-82690-P and through the ''Severo Ochoa Programme for Centres of Excellence in R\&{}D'' (SEV--2015--0554)

%%%%%%%%%%

\end{document}